\documentclass[10pt]{amsart}
\usepackage{amsfonts,amssymb,amscd,amsmath,enumerate,verbatim,calc,amsthm}
\input xy
\xyoption{all}

\headsep=1cm \numberwithin{equation}{section}
\newtheorem{thm}{Theorem}[section]
\newtheorem{cor}[thm]{Corollary}
\newtheorem{lem}[thm]{Lemma}
\newtheorem{prop}[thm]{Proposition}
\newtheorem{defn}[thm]{Definition}
\newtheorem{defns}[thm]{Definitions}

\newtheorem{exam}[thm]{Example}
\newtheorem{rem}[thm]{Remark}
\newtheorem{ques}[thm]{Question}

\newtheorem{defr}[thm]{Definition and Facts }

\DeclareMathOperator{\Supp}{Supp}
\DeclareMathOperator{\Spec}{Spec}
\DeclareMathOperator{\Ass}{Ass}

\newcommand{\Tot}{\mbox{Tot}\,}
\newcommand{\gr}{\mbox{grade}\,}
\def\id{\operatorname{\mathsf{id}}}
\def\Gid{\operatorname{\mathsf{Gid}}}
\def\Gfd{\operatorname{\mathsf{Gfd}}}
\def\Gpd{\operatorname{\mathsf{Gpd}}}

\def\GCpd{\operatorname{\mathsf{G_C-pd}}}
\def\GCfd{\operatorname{\mathsf{G_C-fd}}}
\newcommand{\xid}{\mbox{$\mathcal{X}$-id}\,}
\newcommand{\xpd}{\mbox{$\mathcal{X}$-pd}\,}

\def\pd{\operatorname{\mathsf{pd}}}
\def\fd{\operatorname{\mathsf{fd}}}

\def\cd{\operatorname{\mathsf{cd}}}
\def\cid{\operatorname{\mathsf{C-id}}}
\def\cpd{\operatorname{\mathsf{C-pd}}}

\def\T{\operatorname{\mathsf{T}}}
\def\cgid{\operatorname{\mathsf{G_{C}-id}}}

\newcommand{\h}{\mbox{ht}\,}

\newcommand{\E}{\mbox{E}}

\renewcommand{\H}{\mbox{H}}
\newcommand{\V}{\mbox{V}}

\newcommand{\fa}{\mathfrak{a}}

\newcommand{\fm}{\mathfrak{m}}
\newcommand{\fp}{\mathfrak{p}}

\newcommand{\fn}{\mathfrak{n}}

\newcommand{\C}{\mbox{C}}

\def\depth{\operatorname{\mathsf{depth}}}

\def\Hom{\operatorname{\mathsf{Hom}}}
\def\dim{\operatorname{\mathsf{dim}}}

\def\Ext{\operatorname{\mathsf{Ext}}}
\def\Tor{\operatorname{\mathsf{Tor}}}

\begin{document}
%%% -----------------------------------
\title[flat and gorenstein flat dimensions]
 { on flat and gorenstein flat dimensions of local cohomology modules}

%%% -----------------------------------
%%% -----------------------------------
%%% -----------------------------------
\bibliographystyle{99}
%%% -----------------------------------
%%% -----------------------------------
%%% -----------------------------------

     \author[M.R. Zargar]{Majid Rahro Zargar }
     \author[H. Zakeri]{Hossein Zakeri }

\address{Majid Rahro Zargar\\ School of Mathematics, Institute for Research in Fundamental Sciences (IPM), P.O. Box: 19395-5746, Tehran, Iran.}
\email{zargar9077@gmail.com}
\address{Hossein Zakeri, Faculty of mathematical sciences and computer, Kharazmi
University, 599 Taleghani Avenue, Tehran 15618, Iran.}
\email{zakeri@khu.ac.ir}


\begin{abstract} Let $\fa$ be an ideal of a Noetherian local ring $R$ and let $C$ be a semidualizing $R$-module. For an $R$-module $X$, we denote any of the quantities $\fd_R X$, $\Gfd_R X$ and $\GCfd_RX$ by $\T(X)$. Let $M$ be an $R$-module such that $\H_{\fa}^i(M)=0$ for all $i\neq n$. It is proved that if $\T(X)<\infty$, then $\T(\H_{\fa}^n(M))\leq\T(M)+n$ and the equality holds whenever $M$ is finitely generated. With the aid of these results, among other things, we characterize Cohen-Macaulay modules, dualizing modules and Gorenstein rings.
\end{abstract}

%%% -----------------------------------
%%% -----------------------------------
%%% -----------------------------------

\maketitle

%%% -----------------------------------
\section{introduction}Throughout this paper, $R$ is a commutative Noetherian ring, $\fa$ is an ideal of $R$ and $M$ is an $R$-module. From section 3, we assume that $R$ is local with maximal ideal $\fm$. In this case, $\hat{R}$ denotes the $\fm$-adic completion of $R$ and $\E(R/\fm)$ denotes the injective hull of the residue field $R/\fm$. For each non-negative integer $i$, we use $\H_{\fa}^{i}(M)$ to denote the $i$-th local cohomology module of $M$ with respect to $\fa$ (see \cite{MB} for its definition and basic results). Also, we use $\id_{R}M$, $\pd_{R}M$ and $\fd_{R}M$, respectively, to denote the usual injective, projective and flat dimensions of $M$ respectively. The notions of Gorenstein injective, Gorenstein projective and Gorenstein flat, were introduced by Enochs and Jenda in \cite{EE}. Notice that, the classes of Gorenstein injective, Gorenstein projective and Gorenstein flat modules include, respectively, the classes of injective, projective and flat modules. Recently, the authors proved, in \cite[Theorem 2.5]{MZ}, that if $M$ is a certain module over a local ring $R$, then $\id_{R}M$ and $\id_{R}\H_{\fa}^{\small{\h_{M}\fa}}(M)$ are simultaneously finite and the equality $\id_{R}\H_{\fa}^{\footnotesize{\h_{M}\fa}}(M)=\id_{R}M-\h_{M}\fa$ holds. Also, a counterpart of this result was established in Gorenstein homological algebra. Indeed, it was proved that if $R$ has a dualizing complex and $\Gid_{R}M<\infty$, then the equality $\Gid_{R}\H_{\fa}^{\footnotesize{\h_{M}\fa}}(M)=\Gid_{R}M-\h_{M}\fa$ holds.

The principal aim of this paper is to study, in like manner, the flat (resp. Gorenstein flat) dimension of certain $R$-modules in terms of flat (resp. Gorenstein flat) dimension of their local cohomology modules.

The organization of this paper is as follows. As an our first main result, it is proved, in 3.2. that if $\H_{\fa}^i(M)=0$ for all $i\neq n$, then $\fd_{R}\H_{\fa}^n(M)\leq\fd_{R}M+n$ and the equality holds whenever $M$ is finitely generated. Next, using the above result, we prove, in 3.5, that a $d$-dimensional finitely generated $R$-module $M$ with finite projective dimension is Cohen-Macaulay if and only if $\fd_R\H_{\fm}^d(M)=\pd_RM+d$. Notice that this result recovers \cite[Corollary 9.5.22]{EE}. Propositions 3.7 and 3.9, which provide, respectively, characterization of dualizing modules and Gorenstein rings, recover some results that have been proved in \cite{MR} and \cite{MZ}. It is well known that a local ring $R$ is Cohen-Macaulay if it admits a finitely generated $R$-module $M$ with $\pd_RM<\infty$. In 3.10, we recover this result, by using the assumption $\cpd_{R}M<\infty$ instead of the assumption $\pd_R M<\infty$, where $C$ is a semidualizing $R$-module. Theorem 4.3, which is an another main result, provides a Gorenstein flat version of 3.2. Next, with the aid of this result, over a Cohen-Macaulay local ring, a Gorenstein flat version of 3.5 is established. Finally, again with the aid of 4.3, we obtain a $\GCfd$ version of 3.2. Indeed, we show that if $\H_{\fa}^i(M)=0$ for all $i\neq n$ and $C$ is a semidualizing $R$-module such that
 $\GCfd_{R}M<\infty$, then $\GCfd_{R}\H_{\fa}^{n}(M)\leq\GCfd_{R} M+n$ and the equality holds whenever $M$ is finitely generated. As a generalization of 3.5, this result provides a characterization of Cohen-Macaulay modules in terms of $\GCfd$ dimension of certain local cohomology modules.
\section{preliminaries}

In this section we recall some definitions and facts which are needed throughout this paper.
\begin{defn}\emph{Following \cite[Definition 2.1]{WY}, let $\mathcal{X}$ be a class of $R$-modules and let $M$ be an $R$-module. An $\mathcal{X}$-\textit{coresolution}
of $M$ is a complex of $R$-modules in $\mathcal{X}$ of the form $$X=0\longrightarrow X_{0}\stackrel{\partial_{0}^X} \longrightarrow X_{-1}\stackrel{\partial_{-1}^X}\longrightarrow \cdots\stackrel{\partial_{n+1}^X}\longrightarrow X_{n}\stackrel{\partial_{n}^X}\longrightarrow X_{n-1}\stackrel{\partial_{n-1}^X}\longrightarrow \cdots$$
such that $\H_{0}(X)\cong M$ and {$\H_{n}(X)=0$} for all $n\leq-1$. The $\mathcal{X}$-\textit{injective dimension} of
$M$ is the quantity
$${\xid_{R}M}=\inf\{ \sup \{-n\geq0 | X_{n}\neq0\} ~|~ X~\text{is an $\mathcal{X}$-coresolution of  $M$~}\}.$$
The modules of $\mathcal{X}$-injective dimension zero are precisely the non-zero modules of $\mathcal{X}$ and also $\xid_{R}0=-\infty$.}

\emph{Dually, an $\mathcal{X}$-\textit{resolution} and $\mathcal{X}$-\textit{projective dimension} of
$M$ is defined. We will use the notation $\xpd_{R}M$ to denote the $\mathcal{X}$-projective dimension of $M$.}

\emph{The following notion of semidualizing modules goes back at least to
Vasconcelos \cite{WV}, but was rediscovered by others. The reader is referred to \cite{W} for more details about semidualizing modules.}

\end{defn}
\begin{defn}\emph{A finitely generated $R$-module $C$ is called \textit{semidualizing} if
the natural homomorphism $R\rightarrow \Hom_{R}(C,C)$ is an isomorphism and
$\Ext_{R}^{i}(C,C)=0$ for all $i\geq1$. An $R$-module $D$ is said to be a \textit{dualizing} $R$-module if it is semidualizing and
has finite injective dimension. For a semidualizing $R$-module $C$, we set
{ \[\begin{array}{rl}
&\mathcal{I}_{C}(R)=\{~\Hom_{R}(C,I) | ~~ I  ~~ \text{is an injective $R$-module}\},\\
&\mathcal{P}_{C}(R)=\{~C\otimes_{R}P | ~~ P  ~~ \text{is a projective $R$-module}\},\\
&\mathcal{F}_{C}(R)=\{~C\otimes_{R}F | ~~ F  ~~ \text{is a flat $R$-module}\}.
\end{array}\]}
The $R$-modules in $\mathcal{I}_{C}(R)$, $\mathcal{P}_{C}(R)$ and $\mathcal{F}_{C}(R)$ are called $C$-injective, $C$-projective and $C$-flat, respectively.
For convenience the quantities $\mathcal{I_{C}}(R)$-$\id_{R}M$ and $\mathcal{P_{C}}(R)$-$\pd_{R}M$, which are defined as in 2.1, are denoted by $\cid_{R}M$ and $\cpd_{R}M$ respectively.
Notice that if $C=R$, then the above quantities are the usual injective and projective dimensions, respectively. }
\end{defn}

Based on the work of E.E. Enochs and O.M.G. Jenda \cite{EE}, the following notions were introduced
and studied by H. Holm and P. J${\o}$rgensen \cite{HJ}.
\begin{defns}\emph{Let $C$ be a semidualizing $R$-module. A complete $\mathcal{I}_{C}\mathcal{I}$-\textit{coresolution} is a complex $Y$ of $R$-modules such that
\begin{itemize}
\item[(i)]{$Y$ is exact and {$\Hom_{R}(I,Y)$} is exact for each $I\in \mathcal{I}_{C}(R)$, and that}
\item[(ii)]{ $Y_{i}\in \mathcal{I}_{C}(R)$ for all $i>0$ and $Y_{i}$ is injective for all $i\leq0$.}
\end{itemize}
An $R$-module $M$ is called $G_{C}$-\textit{injective} if there exists a complete $\mathcal{I}_{C}\mathcal{I}$-coresolution $Y$
such that $M\cong \ker(\partial_{0}^{Y})$. In this case $Y$ is a complete $\mathcal{I}_{C}\mathcal{I}$-coresolution of $M$. The class of $G_{C}$-injective $R$-modules is denoted by $\mathcal{GI_{C}}(R)$ and for convenience, the quantity $\mathcal{GI_{C}}(R)$-$\id_{R}M$ which is defined as in 2.1, is denoted by $\cgid_{R}M.$}

\emph{Dually, we can define the notions of $G_{C}$-\textit{projective} and $G_{C}$-\textit{flat} dimensions  for an $R$-module $M$ which are denoted by $\GCfd_R M$ and $\GCpd_R M$, respectively. For more details, the reader is refereed to \cite[Definiation 2.7]{HJ}. Note that when $C=R$, these notions are exactly the concepts of Gorenstein injective, Gorenstein projective and Gorenstein flat dimensions which were introduced in \cite{EE}.}
\end{defns}

\begin{defn} \emph{We say that a finitely generated $R$-module $M$ is \textit{relative Cohen–
Macaulay with respect to $\fa$} if there is precisely one non-vanishing local cohomology module of $M$ with
respect to $\fa$. Clearly this is the case if and only if $\gr(\fa,M)=\cd(\fa,M)$, where $\cd(\fa,M)$ is the largest integer $i$ for which $\H_{\fa}^i(M)\neq0$ and $\gr(\fa,M)$ is the least integer $i$ such that \emph{$\Ext_R^i({R/\fa},M)\neq0$}. Observe that the notion of relative Cohen-Macaulay module is connected with the notion of cohomologically complete intersection ideal which has been studied in \cite{HSH}.
}
\end{defn}
\begin{rem}\emph{Let $M$ be a relative Cohen-Macaulay module with respect to $\fa$ and let $\cd(\fa,M)=n$. Then, in view of \cite[theorems 6.1.4, 4.2.1, 4.3.2]{MB}, it is easy to see that $\Supp\H^{n}_{\fa}(M)=\Supp({M}/{\fa M})$ and $\h_{M}\fa=\gr(\fa,M)$, where
$\h_{M}\fa =\inf\{\ \dim_{R_{\fp}}M_{\fp} |~ \fp\in\Supp(M/\fa M) ~\}$.}
\end{rem}
Next, we recall some elementary results about the trivial extension of a ring by a module.

\begin{defr}\emph{Let $C$ be an $R$-module. Then the direct sum $R\oplus C$ has the structure of a commutative ring with respect to the multiplication defined by
$$(a,c)(a',c')=(aa', ac'+a'c),$$}
\emph{for all $(a,c),(a', c')$ of $ R\oplus C$. This ring is called the \textit{trivial extension} of $R$ by $C$ and is denoted by $R\ltimes C$. The following properties of $R\ltimes C$ are needed in this paper.}
\begin{itemize}
\item[(i)]\emph{{There are natural ring homomorphisms $R\rightleftarrows R\ltimes C$ which enable us to consider $R$-modules as $R\ltimes C$-modules, and vice versa.}}
\item[(ii)]\emph{{For any ideal $\fa$ of $R$, $\fa\oplus C$ is an ideal of $R\ltimes C$.}
\item[(iii)]{$(R\ltimes C, \fm\oplus C)$ is a Noetherian local ring whenever $(R,\fm)$ is a Noetherian local ring and $C$ is a finitely generated $R$-module. Also, in this case, $\dim R=\dim R\ltimes C$}.}
\end{itemize}
\end{defr}
The classes defined next is collectively known as Foxby classes. The reader is referred to \cite{W} for some basic results about those classes.
\begin{defn}\emph{Let $C$ be a semidualizing $R$-module. The \textit{Bass class} with respect to $C$ is the class $\mathcal{B_{C}}(R)$ of $R$-modules $M$ such that }
\begin{itemize}
\item[(i)]{\emph{$\Ext_{R}^i(C,M)=0=\Tor^{R}_{i}(C,\Hom_{R}(C,M))$ for all $i\geq1$}, and that}
\item[(ii)]{\emph{the natural evaluation map $C\otimes_{R}\Hom_{R}(C,M)\rightarrow M$ is an isomorphism}.}
\end{itemize}
\emph{Dually, the \textit{Auslander class} with respect to $C$, denoted by $\mathcal{A}_{C}(R)$, consists of
all $R$-modules $M$ such that}
\begin{itemize}
\item[(i)]{\emph{$\Tor^{R}_{i}(C,M)=0=\Ext_{R}^{i}(C,C\otimes_{R}M)$ for all $i\geq1$}, and that}
\item[(ii)]{\emph{the natural map $M\rightarrow \Hom_{R}(C,C\otimes_{R}M)$ is an isomorphism}.}
\end{itemize}
\end{defn}
\section{local cohomology and flat dimension}
The starting point of this section is the following proposition which plays an essential role in the present paper.
\begin{prop}Let $n$ and $s$ be non-negative integers and let $N$ be an $R$-module. Suppose that $\emph{H}^{i}_\fa(M)=0$ for all $i\neq n$. Then the following statements hold true.
\begin{itemize}
\item[(i)]{ If \emph{$\Tor_i^R(N,M)=0$} for all $i>s$, then \emph{$\Tor_i^R(N,\H_{\fa}^n(M))=0$} for all $i>s+n$.}
\item[(ii)]{If $N$ is $\fa$-torsion, then  \emph{$\Tor_i^R(N,\H_{\fa}^n(M))\cong\Tor_{i-n}^R(N,M)$ for all $i$.}}
\end{itemize}
\end{prop}
\begin{proof} (i): We may assume $\H^{n}_\fa(M)\neq0$. Let $c$ be the arithmetic rank of $\fa$. Then there exists a sequence $x_{1},\ldots,x_{c}$ of elements of $R$ such that $\sqrt{\fa} = \sqrt{(x_{1},\ldots,x_{c})}$. Let $C(R)^\bullet$ denotes the $\check{C}$ech
complex of $R$ with respect to $x_{1},\ldots, x_{c}$ and let $F_{\bullet}$ be a free resolution for $N$. For the first quadrant bicomplex $\mathcal{M}=\{ M_{p,q}=F_{p}\otimes_{R}M\otimes_{R}C_{c-q}\}$ we denote the total complex of $\mathcal{M}$ by $\Tot(\mathcal{M})$. Now, with the notation of \cite{JR}, $\E^1$ is the bigraded module whose $(p,q)$ term is $\H^{''}_{q}(M_{p,*})$, the q-th homology of the p-th column. Since $F_{p}$ is flat, by assumption we have
\[ ^{I}E_{p,q}^{1}=\H^{''}_{q}(M_{p,*})=\begin{cases}
       0 & \text{if $q\neq c-n$}\\
       F_{p}\otimes_{R}\H^n_{\fa}(M) & \text{if $q=c-n$},
       \end{cases} \]
therefore

       \[ ^{I}E_{p,q}^{2}=\H^{'}_{p}\H^{''}_{q}(\mathcal{M})=\begin{cases}
       0 & \text{if $q\neq c-n$}\\
       \Tor^{R}_{p}(N, \H^n_{\fa}(M)) & \text{if $q=c-n$};
       \end{cases} \]
       and hence the spectral sequence collapses. Note that, in view of \cite[Theorem 10.16]{JR} we have $^{I}E_{p,q}^{2}\underset{p}\Longrightarrow\H_{p+q}(\Tot(\mathcal{M}))$ for all $p,q$. Thus, for all $t=p+q$, there is the following filtration  $${0}=\Phi^{-1}H_{t}\subseteq\Phi^{0}H_{t}\subseteq\ldots\subseteq\Phi^{t-1}H_{t}\subseteq\Phi^{t}H_{t}=H_{t}$$
       such that $^{I}E_{p,q}^{\infty}\cong\Phi^{p}H_{t}/\Phi^{p-1}H_{t}$. Therefore, one can use the above filtration to see that \begin{equation}\Tor^{R}_{p}(N, \H^n_{\fa}(M))\cong\H_{p+c-n}(\Tot(\mathcal{M}))\end{equation} for all $p$.

       A similar argument applies to the second iterated homology, using the fact that each $C_{c-q}$ is flat, yields  { \[ ^{II}E_{p^{'},q^{'}}^{2}=\H^{''}_{p^{'}}\H^{'}_{q^{'}}(\mathcal{M})=\begin{cases}
       0 & \text{if $q^{'}> s$}\\
       \H_{\fa}^{c-p^{'}}(\Tor^{R}_{q^{'}}(N,M)) & \text{if $q^{'}\leq s.$}
       \end{cases} \]}

       Now, we claim that $^{II}E_{p^{'},q^{'}}^{\infty}=0$ for all $p^{'}, q^{'}$ such that $p^{'}+q^{'}=p+c-n$ and that $p>s+n$. To this end, first notice that, by \cite[Theorem 10.16]{JR}, we have $^{II}E_{p^{'},q^{'}}^{2}\underset{p_{'}}\Longrightarrow\H_{p^{'}+q^{'}}(\Tot(\mathcal{M}))$. If $q^{'}>s$, there is nothing to prove. Let $q^{'}\leq s$. Then $0>c-p^{'}$ and hence $^{II}E_{p^{'},q^{'}}^{2}=0$; which in turn yields $^{II}E_{p^{'},q^{'}}^{\infty}=0$. Now, by using a similar filtration as above, one can see that $\H_{p+c-n}(\Tot(\mathcal{M}))=0$  for all $p>s+n$. Therefore $\Tor^{R}_{p}(N, \H^n_{\fa}(M))=0$ for all $p>s+n$.

(ii): First, notice that $\Tor^{R}_{i}(N,M)$ is an $\fa$-torsion $R$-module for all $i$. Therefore, by using the same arguments as above, one can deduce that
\[ ^{II}E_{p^{'},q^{'}}^{2}=\H^{''}_{p^{'}}\H^{'}_{q^{'}}(\mathcal{M})=\begin{cases}
       0 & \text{if $p^{'}\neq c$}\\
       \Tor^{R}_{q^{'}}(N,M) & \text{if $p^{'}=c$.}
       \end{cases} \]
Thus, the spectral sequence collapses at the c-th column; and hence we get the isomorphism $\Tor^{R}_{q^{'}}(N, M)\cong\H_{q^{'}+c}(\Tot(\mathcal{M}))$ for all $q^{'}$. It therefore follows, by the isomorphism (3.1), that $$\Tor_{p}^{R}(N,\H_{\fa}^{n}(M))\cong\Tor_{p-n}^{R}(N,M)$$
for all $p$.\end{proof}

The following theorem, which is one of the main results of this section, provides a comparison between the flat dimensions of a relative Cohen-Macaulay  module and its non-zero local cohomology module. Here we adopt the convention that the flat dimension of the zero module is to be taken as $-\infty$.
\begin{thm}Let $n$ be a non-negative integer such that $\emph{H}^{i}_\fa(M)=0$ for all $i\neq n$. Then
\begin{itemize}
\item[(i)]{\emph{$\fd_{R}\H^{n}_\fa(M)\leq\fd_R M+n$, and}}
\item[(ii)]{the equality holds whenever $M$ is finitely generated.}
\end{itemize}
\begin{proof}(i) follows immediately from Proposition 3.1(i). It is well-known, see for example \cite[Theorem 8.27]{JR}, that $\pd_R M=\fd_R M$ whenever $M$ is finitely generated. Therefore, one can use \cite[Corollary 8.54]{JR} in conjunction with Proposition 3.1(ii) and the inequality (i) to establish the final assertion.
\end{proof}
\end{thm}
{Next, we provide an example to show that if $M$ is not finitely generated, then Theorem 3.2(ii) is no longer true.}
\begin{exam}\emph{Let $k$ be a field and let $R={k[[x,y,z]]}/{(x^2,xy)}$. Set $\fp=(x,y)R$. Notice that $R_{\fp}$ is not Gorenstein, $\fp\notin\V(zR)$ and $R$ is relative Cohen-Macaulay with respect to $zR$. Set $M=R\oplus\E(R/\fp)$. Now, since $R$ is a local ring with $\dim R =2$, $M$ is not finitely generated. Note that $\H_{zR}^{i}(M)=0$ for all $i\neq1$ and $\H_{zR}^{1}(M)\cong\H_{zR}^{1}(R)$. Therefore, one can use Theorem 3.2 to see that $\fd_{R}\H_{zR}^{1}(M)=1$. On the other hand, since $R_{\fp}$ is not Gorenstein, $\fd_{R}\E(R/\fp)=\infty$; and hence $\fd_{R}M=\infty.$}
\end{exam}

The next corollary shows that the equality in Theorem 3.2(i) may happen even if $M$ is not finitely generated.
\begin{cor}Suppose that $R$ is relative Cohen-Macaulay with respect to $\fa$ and that \emph{$\h_{R}\fa=n$}. Then, for every non-zero faithfully flat $R$-module $M$ we have \emph{$\fd_{R}\H_{\fa}^{n}(M)=n$.}
\begin{proof}Let $M$ be a non-zero faithfully flat $R$-module. Since the functor $\H_{\fa}^{n}(-)$ is right exact, we have $\H_{\fa}^{n}(M)\cong\H_{\fa}^{n}(R)\otimes_{R}M$; and hence by assumption $\H_{\fa}^{n}(M)\neq0$ and $\fm M\neq M$. By, \cite[Theorem 5.40]{JR}, there is a directed index set $I$ and a family of finitely generated free $R$-modules $\{ M_{i}\}_{i\in I}$ such that $M=\underset{{i\in I}}\varinjlim M_{i}$. Notice that each $M_{i}$ is relative Cohen-Macaulay with respect to $\fa$ and that $\h_{M_{i}}\fa=n$. Therefore, $\H_{\fa}^{j}(M)=\underset{i\in I}\varinjlim\H_{\fa}^{j}(M_{i})=0$ for all $j\neq n$; and hence, in view of Theorem 3.2(i), we get $\fd_{R}\H_{\fa}^{n}(M)\leq n$. Now, if $\fd_{R}\H_{\fa}^{n}(M)<n$, then $\Tor^{R}_{n}(R/\fm,\H_{\fa}^{n}(M))=0$. But, by Proposition 3.1(ii), $\Tor_{n}^{R}(R/\fm,\H_{\fa}^{n}(M))\cong M/\fm M\neq 0$ which is a contradiction.
\end{proof}
\end{cor}
The next proposition is a generalization of \cite[Proposition 9.5.22]{EE}.
\begin{prop}Let $M$ be a $d$-dimensional finitely generated $R$-module of finite projective dimension. Then the following statements are equivalent.
\begin{itemize}
\item[(i)]{\emph{$M$ is Cohen-Macaulay.}}
\item[(ii)]{\emph{$\fd_{R}\H_{\fm}^d(M)=\fd_{R}M+d$}}.
\item[(iii)]{\emph{$\pd_{R}\H_{\fm}^d(M)=\pd_{R}M+d$}}.
\end{itemize}
\begin{proof}We first notice that the Artinian $R$-module $\H_{\fm}^{d}(M)$ has a natural $\hat{R}$-module structure and that $\fd_{R}\H_{\fm}^d(M)=\fd_{\hat{R}}\H_{\fm}^d(M)$. Now, assume that $\fd_{R}\H_{\fm}^d(M)<\infty$. Then, in view of \cite[Proposition 6]{J} and \cite[Theorem 3.2.6]{GM}, we see that $\fd_{R}\H_{\fm}^d(M)\leq\pd_{R}\H_{\fm}^d(M)\leq\dim R$. Next, by \cite[Theorem 3.1.17]{BH}, \cite[Theorem 4.16]{CF} and the Bass's theorem, one can deduce that$$\fd_{\hat{R}}\H_{\fm}^d(M)=\id_{\hat{R}}\Hom_{\hat{R}}(\H_{\fm}^d(M),\E_{\hat{R}}(\hat{R}/\fm\hat{R}))=\depth \hat{R}=\dim R.$$ It therefore follows that $\fd_{R}\H_{\fm}^d(M)=\pd_{R}\H_{\fm}^d(M)=\dim R$ and that $R$ is Cohen-Macaulay.

Now, the implications (ii)$\Leftrightarrow$(iii) follows immediately from the above argument.

(ii)$\Rightarrow$(i): Since $\fd_{R}\H_{\fm}^d(M)<\infty$, one can use the conclusion of the above argument in conjunction with the Auslander-Buchsbaum Theorem \cite[Theorem 1.3.3]{BH} to see that $M$ is Cohen-Macaulay. Finally the implication (i)$\Rightarrow$(ii) follows from Theorem 3.2.
\end{proof}
\end{prop}
Let $(R,\fm)$ be a local ring and let $M$ be a finitely generated $R$-module with finite projective dimension. It follows from Theorem 3.2 that if $M$ is relative Cohen-Macaulay with respect to an ideal $\fa$ of $R$, then $\fd_{R}\H_{\fa}^{\cd(\fa,M)}(M)=\pd_{R}M+\cd(\fa,M)$. Also, in pervious proposition, we deuced that the converse holds whenever $\fa=\fm$. Therefore, it is natural to ask the following question.
\begin{ques}Let $M$ be a finitely generated $R$-module of finite projective dimension and let $\fa$ be a non-maximal ideal of $R$. Are the following statements equivalent?
\begin{itemize}
\item[(i)]{\emph{$M$ is relative Cohen-Macaulay with respect to $\fa$.}}
\item[(ii)]{\emph{$\fd_{R}\H_{\fa}^{\cd(\fa,M)}(M)=\pd_{R}M+\cd(\fa,M)$}}.
\end{itemize}
\end{ques}
The next proposition has been proved in \cite[Proposition 3.3]{MR} under the extra conditions that the underlying ring is Cohen-Macaulay and admits a dualizing complex.
\begin{prop} Let $C$ be a semidualizing $R$-module. Then the following statements are equivalent.
\begin{itemize}
\item[(i)]{$C$ is a dualizing $R$-module.}
\item[(ii)]{\emph{$\cgid_{R}\H_{\fa}^n(R)<\infty$ for all ideals $\fa$ of $R$ such that $R$ is relative Cohen-Macaulay with respect to $\fa$ and that $\h_{R}\fa=n.$ }}
\item[(iii)]{\emph{$\cgid_{R}\H^n_{\fa}(R)<\infty$ for some ideal $\fa$ of $R$ such that $R$ is relative Cohen-Macaulay with respect to $\fa$ and that ${\h}_{R}\fa =n.$}}
\end{itemize}
\begin{proof}The implication (i)$\Rightarrow$(ii) follows from \cite[Theorem 3.2(ii)]{MR} and the implication (ii)$\Rightarrow$(iii) is clear.

(iii)$\Rightarrow$(i): Suppose that $\cgid_{R}\H^n_{\fa}(R)<\infty$, where $\fa$ is an ideal of $R$ such that $R$ is relative Cohen-Macaulay with respect to $\fa$ and that $\h_{R}\fa=n$.
Then, in view of Theorem 3.2, $\fd_{R}\H_{\fa}^n(R)<\infty$. Hence, one can use \cite[Proposition 6]{J} to see that $\pd_{R}\H_{\fa}^n(R)<\infty$. Therefore, by \cite[Theorem 2.3]{WY}, we have $\cgid_{R}\H^n_{\fa}(R)=\cid_{R}\H^n_{\fa}(R)$. Hence, one can use \cite[Theorem 3.2(ii)]{MR} to complete the proof.
\end{proof}
\end{prop}
An immediate consequence of the previous proposition is the next Corollary, which has been proved in \cite[Corollary 3.10]{MZ} under the additional assumptions that $R$ is Cohen-Macaulay and admits a dualizing complex.
\begin{cor}The following statements are equivalent.
\begin{itemize}
\item[(i)]{\emph{$R$ is a Gorenstein ring.}}
\item[(ii)]{\emph{$\Gid_{R}\H_{\fa}^n(R)<\infty$ for all ideals $\fa$ of $R$ such that $R$ is relative Cohen-Macaulay with respect to $\fa$ and that $\h_{R}\fa=n.$ }}
\item[(iii)]{\emph{$\Gid_{R}\H_{\fa}^n(R)<\infty$ for some ideal $\fa$ of $R$ such that $R$ is relative Cohen-Macaulay with respect to $\fa$ and that $\h_{R}\fa=n.$ }}
\end{itemize}
\end{cor}
It follows from the proof of \cite[Theorem 3.2(i)]{MR} that if $n$ is a non-negative integer and $M$ is an $R$-module( not necessarily finitely generated) such that ${\H}^{i}_{\fa}(M)=0$ for all $i\neq n$ and that $\cid_{R}M$ is finite, then $\cid_{R}{\H}^{n}_{\fa}(M)$ is finite. This fact leads us to the following proposition which recovers \cite[Theorem 3.8]{MR}.
\begin{prop} Let $C$ be a semidualizing $R$-module. Consider the following statements.
\begin{itemize}
\item[(i)]{$R$ is Gorenstein.}
\item[(ii)]{\emph{$\cid_{R}{\H}^{n}_{\fa}(C)<\infty$} for all ideals $\fa$ of $R$ such that $R$ is relative Cohen-Macaulay with respect to $\fa$ and that $\emph{\h}_{R}\fa =n$}.
\item[(iii)]{\emph{$\cid_{R}{\H}^{n}_{\fa}(C)<\infty$} for some ideal $\fa$ of $R$ such that $R$ is relative Cohen-Macaulay with respect to $\fa$ and that $\emph{\h}_{R}\fa =n$}.
\end{itemize}
\emph{Then, the implications (i)$\Rightarrow$(ii)$\Rightarrow$(iii) hold true, while (iii) implies (i) whenever $R$ is Cohen-Macaulay.}
\begin{proof}
First, notice that $R\cong C$ whenever $R$ is Gorenstein. Hence, the implication (i)$\Rightarrow$(ii) follows from \cite[Theorem 2.5(i)]{MZ} and the implication (ii)$\Rightarrow$(iii) is clear.

(iii)$\Rightarrow$(i): Let $\fa$ be an ideal of $R$ such that $R$ is relative Cohen-Macaulay with respect to $\fa$ and that $\h_{R}\fa=n$. Since $\Supp_{R}(C)=\Spec(R)$, in view of \cite[Theorem 2.2]{DNT}, we get $\cd(\fa,R)=\cd(\fa,C)$. On the other hand, by \cite[Theorem 2.2.6(c)]{W}, $\gr(\fa,R)=\gr(\fa,C)$. Hence, 2.4 implies that $C$ is relative Cohen-Macaulay with respect to $\fa$. Since $\H_{\fm}^0(\E(R/\fm))=\E(R/\fm)$ and for any non-maximal prime ideal $\fp$ of $R$, the $R$-module $\H_{\fm}^0(\E(R/\fp))=0$ vanishes, we may apply \cite[Proposition 2.8]{MZ} to see that $\H_{\fm}^i(\H_{\fa}^n(C))=\H_{\fm}^{n+i}(C)$ for all $i\geq0$. Therefore, by considering the additional assumption that $R$ is Cohen-Macaulay, one can deduce that
\[ \H_{\fm}^i(\H_{\fa}^n(C))=\begin{cases}
       0 & \text{if $i\neq\dim R/\fa $}\\
       \H^d_{\fm}(C) & \text{if $i=\dim R/\fa ,$}
       \end{cases} \]
where $d=\dim R$. Thus, by the assumption and \cite[Theorem 3.2(i)]{MR}, we see that $\cid_{R}{\H}^{d}_{\fm}(C)$ is finite. Now, one can use \cite[Theorem 3.8]{MR} to complete the proof.
\end{proof}
\end{prop}
 It is known that if a local ring admits a non-zero Cohen-Macaulay module of finite projective dimension, then it is a Cohen-Macaulay ring. The following theorem is a generalization of this result.
\begin{thm}Let $C$ be a semidualizing $R$-module. If there exists a non-zero Cohen-Macaulay $R$-module $M$ with finite \emph{$\cpd_{R}M$}, then $R$ is Cohen-Macaulay.
\begin{proof}Let $M$ be a non-zero Cohen-Macaulay $R$-module of dimension $n$ such that $\cpd_{R}M$ is finite. Notice that, in view of \cite[Theorem 2.11(c)]{TW}, we have $\cpd_{R}M=\pd_{R}\Hom_{R}(C,M)$. Also, since $C\otimes_{R}\hat{R}$ is a semidualizing $\hat{R}$-module and $\Hom_{\hat R}(\hat{C},\hat{M})\cong\Hom_{R}(C,M)\otimes_{R}\hat R$, we may assume that $R$ is complete. Now, by using \cite[Corollary 2.9(a)]{TW}, we have $M\in\mathcal{B}_{C}(R)$. Therefore, $\Tor_{i}^{R}(C,\Hom_{R}(C,M))=0$ for all $i>0$ and $C\otimes_{R}\Hom_{R}(C,M)\cong M$. Hence, one can use \cite[Theorem 1.2]{AUS} to obtain the following equalities{ \[\begin{array}{rl}
 \depth_{R} M&=\depth_{R}(C\otimes_{R}\Hom_{R}(C,M))\\
 &=\depth_{R} C-\depth R+\depth_{R}\Hom_{R}(C,M)\\
 &=\depth_{R}\Hom_{R}(C,M).
 \end{array}\]}
 On the other hand, since $\Ass_{R}(\Hom_{R}(C,M))=\Ass_{R}(M)$ and $M$ is Cohen-Macaulay, we see that $\dim_{R} M=\dim_{R}\Hom_{R}(C,M)$. Therefore, $\Hom_{R}(C,M)$ is Cohen-Macaulay. Hence, one can use Theorem 3.2 to see that the injective dimension of the finitely generated $R$-module $\Hom_{R}(\H_{\fm}^{n}(\Hom_{R}(C,M)),\E_{R}(R/\fm))$ is finite. Therefore, by the Bass's theorem, $R$ is Cohen-Macaulay.
\end{proof}
\end{thm}
Applying Theorem 3.10 to the semidualizing $R$-module $C=R$,
we immediately obtain the following well-known result.
\begin{cor}If $R$ admits a non-zero Cohen-Macaulay module of finite projective dimension, then $R$ is Cohen-Macaulay.
\end{cor}
\section{local cohomology and gorenstein flat dimension}

The starting point of this section is the next lemma which has been proved, in \cite[Lemma 3.7]{MZ} and \cite[Corollary 3.9]{MZ}, under the extra assumption that $R$ is Cohen-Macaulay.
\begin{lem}Suppose that $M$ is a non-zero finitely generated $R$-module. Then the following statements hold true.
\begin{itemize}
\item[(i)]{\emph{Suppose that $x\in\fm$ is both $R$-regular and $M$-regular. Then $\Gid_{R}M <\infty$ if and only if {$\Gid_{{R/xR}}(M/xM)<\infty$.}}}
\item[(ii)]{\emph{Assume that $M$ is Cohen-Macaulay of dimension $n$. Then $\Gid_{R}M<\infty$ if and only if $\Gid_{R}\H_{\fm}^n(M)<\infty$.}}
\end{itemize}
\begin{proof}
First notice that, by \cite[Theorem 3.24]{CF}, $\Gid_{R}M=\Gid_{\hat{R}}\hat{M}$. On the other hand, since $\H_{\fm}^{i}(M)$ is Artinian, in view of \cite[Lemma 3.6]{RS}, we have $\Gid_{R}\H_{\fm}^n(M)=\Gid_{\hat{R}}\H_{\fm}^n(M)=\Gid_{\hat{R}}\H_{\fm\hat{R}}^n(\hat{M})$. Thus, we can assume that $R$ is complete; and hence it has a dualizing complex $D$.

(i): Set $\overline{R}=R/xR$. We notice that $\fd_{R}\overline{R}<\infty$ and ${\mu}^{i+\small{\depth R}}(\fm,R)=\mu^{i+\small{\depth \overline{R}}}(\overline{\fm},\overline{R})$ for all $i\in \mathbb{Z}$, where $\mu^{i}(\fm,R)$ denotes the i-th Bass number of $R$ with respect to $\fm$. Hence, by using \cite[2.11]{AV}, we see that $D\otimes_{R}^{\mathbf{L}}\overline{R}$ is a dualizing complex for $\overline{R}$. On the other hand, by the assumption, one can deduce that $\Tor^{R}_{i}(\overline{R},M)=0$ for all $i>0$. Therefore $\overline{M}\simeq M\otimes_{R}^{\mathbf{L}}\overline{R}$, in derived category $\mathcal{D}(R)$. Now, we can use \cite[Theorem 5.3]{CFH} to complete the proof.

(ii): Let $M$ be Cohen-Macaulay with $\dim M=n$. Then, the implication ($\Rightarrow$) follows from \cite[Theorem 3.8(i)]{MZ}. To prove the converse, we proceed by induction on $n$. The case $n=0$ is obvious. Assume that $n>0$ and that the result has been proved for $n-1$. Now, by using \cite[Theorem 3.12(ii)]{MZ} in conjunction with the assumption, one can choose an element $x$ in $\fm$ which is both $R$-regular and $M$-regular. Next, we can use the induced exact sequence $$0\longrightarrow {\H}^{n-1}_\fm(M/xM)\longrightarrow {\H}^{n}_\fm(M) \longrightarrow {\H}^{n}_\fm(M) \longrightarrow 0$$  and \cite[Proposition 3.9]{CF}  to see that $\Gid_{R}{\H}^{n-1}_\fm(M/xM)$ is finite. Hence, by the inductive hypothesis, $\Gid_{R}M/xM$ is finite. Therefore, in view of \cite[Theorem 7.6(b)]{CF}, $\Gid_{R/xR}M/xM<\infty$. It therefore follows from part (i) that $\Gid_{R}M$ is finite. Now the result follows by induction.

\end{proof}
 \end{lem}
\begin{lem}Suppose that $M$ is a Cohen-Macaulay $R$-module of dimension $n$ such that {\emph{$\Gfd_{R}\H_{\fm}^n(M)$}} is finite. Then \emph{$\Gfd_{R}M$} is finite.
\begin{proof}First notice that, in view of \cite[Theorem 4.27]{CF}, we have $\Gfd_{R}\H_{\fm}^d(M)=\Gfd_{\hat{R}}\H_{\fm\hat{R}}^d(\hat{M})$ and $\Gfd_{R}M=\Gfd_{\hat R}\hat M$. Therefore, without lose of generality, we can assume that $R$ is complete; and hence it is a homomorphic image of a Gorenstein local ring $(S,\fn)$ of dimension $d$. Now, in view of the local duality theorem \cite[Theorem 11.2.6]{MB}, we have \begin{equation}
\H_{\fm}^{n}(M)\cong\Hom_{R}(\Ext_{S}^{d-n}(M,S),\E(R/\fm)).\end{equation} Next, we notice that $M$ is a Cohen-Macaulay $S$-module of dimension $n$; and hence, by \cite[Theorem 3.3.10(c)(i)]{BH}, the $S$-module $\Ext_{S}^{d-n}(M,S)$ is Cohen-Macaulay of dimension $n$. So that it is a Cohen-Macaulay $R$-module. Therefore, again, we can use the local duality theorem and \cite[Theorem 3.3.10(c)(iii)]{BH} to obtain the following isomorphisms\begin{equation}\end{equation}\vspace{-1.13cm}
{\[\begin{array}{rl}
\H_{\fm}^{n}(\Ext_{S}^{d-n}(M,S))&\cong\Hom_{R}(\Ext_{S}^{d-n}(\Ext_{S}^{d-n}(M,S),S),\E(R/\fm))\\
&\cong\Hom_{R}(M,\E(R/\fm)).
\end{array}\]}

Now, by our assumption, (4.1) and \cite[Theorem 4.25]{CF}, we have $\Gid_{R}\Ext_{S}^{d-n}(M,S)<\infty$. Therefore, one can use (4.2), \cite[Theorem 4.16]{CF} and Lemma 4.1(ii) to see that $\Gfd_R M$ is finite.
 \end{proof}
\end{lem}

The following theorem, which is the main result of this section, provides a comparison between the Gorenstein flat dimensions of a relative Cohen-Macaulay  module and its non-zero local cohomology module.
\begin{thm}Let $n$ be a non-negative integer and let $M$ be an $R$-module such that $\emph{H}^{i}_\fa(M)=0$ for all $i\neq n$. Then the following statements hold true.
\begin{itemize}
\item[(i)]{ If \emph{$\Gfd_{R} M<\infty$}, then \emph{$\Gfd_{R}\H^{n}_\fa(M)\leq\Gfd_{R}M+n$.}}
\item[(ii)]{If \emph{$\Gfd_{R}\H^{n}_\fa(M)<\infty$}, then \emph{$\Gfd_R M<\infty$} whenever $M$ is Cohen-Macaulay.}
\end{itemize}
Furthermore, in \emph{(i)} equality holds whenever $M$ is finitely generated.
\end{thm}
\begin{proof}First notice that $\sum^n{\H_{\fa}^n(M)}\simeq \C_{\fa}(R)\otimes_{R}M$, where $\C_{\fa}(R)$ denotes the $\check{C}$ech
complex of $R$ with respect to a generator of $\fa$. Now, assume that $s:=\Gfd_{R}M$ is finite and that $X^{\bullet}$ is a Gorenstein flat resolution for $M$. Then there exists a quasiisomorphism  $X^{\bullet}\overset{\simeq}\longrightarrow M$. Hence, by \cite[Corollary 2.16]{CFH}, $\sum^n{\H_{\fa}^n(M)}\simeq \C_{\fa}(R)\otimes_{R}^{\mathbf{L}}X^{\bullet}$. Since $\C_{\fa}(R)\otimes_{R}^{\mathbf{L}}X^{\bullet}$ is a bounded complex of Gorenstein flat modules, we see that $\Gfd_{R}\H_{\fa}^n(M)$ is finite. Next, by \cite[Theorem 4.17]{CF}, $\Tor_i^R(E,M)=0$ for all $i>s$ and for all injective $R$-modules $E$. Hence, by Proposition 3.1(i), $\Tor_i^R(E,\H_{\fa}^n(M))=0$ for all $i>n+s$ and for all injective $R$-modules $E$. Therefore $\Gfd_{R}\H^{n}_\fa(M)\leq s+n.$

(ii). Suppose that $M$ is Cohen-Macaulay and that $\dim M=d$. Then, by \cite[Proposition 2.8]{MZ}, one can deduce that
\[ \H_{\fm}^i(\H_{\fa}^n(M))=\begin{cases}
       0 & \text{if $i\neq\dim M/\fa M $}\\
       \H^d_{\fm}(M) & \text{if $i=\dim M/\fa M.$}
       \end{cases} \]
Therefore, we can use (i) and Lemma 4.2 to see that $\Gfd_{R}M$ is finite.

For the final assertion, suppose that $M$ is finitely generated with $\Gfd_{R} M=s<\infty$. Then, by (i), $\Gfd_{R}\H^{n}_\fa(M)\leq s+n$. If $\Gfd_{R}\H^{n}_\fa(M)< s+n$, then, in view of \cite[Theorem 4.17]{CF}, we deduce that $\Tor^{R}_{s+n}(\E(k),\H^{n}_{\fa}{(M)})=0$. Hence, by Proposition 3.1(ii), one can see that $\Tor^{R}_{s}(\E(k),M)=0$ which is a contradiction by \cite[Theorem 2.4.5(b)]{CF0} and \cite[Proposition 4.24]{CF}. Therefore, $\Gfd_{R}\H^{n}_\fa(M)=\Gfd_{R}M+n.$
\end{proof}
An immediate consequence of the previous theorem is the following corollary.
\begin{cor}Let $M$ be a Cohen-Macaulay $R$-module of dimension $d$. Then \emph{$\Gfd_{R}\H^{d}_\fm(M)=\Gfd_{R}M+d.$}
\end{cor}

The following proposition is a Gorenstein projective version of Proposition 3.5.
\begin{prop}Assume that $R$ is Cohen-Macaulay and that $M$ is a $d$-dimensional finitely generated $R$-module of finite Gorenstein projective dimension. Then the following statements are equivalent.
\begin{itemize}
\item[(i)]{\emph{$M$ is Cohen-Macaulay. }}
\item[(ii)]{\emph{$\Gfd_{R}\H_{\fm}^d(M)=\Gfd_{R}M+d$}}.
\item[(iii)]{\emph{$\Gpd_{R}\H_{\fm}^d(M)=\Gpd_{R}M+d$}}.
\end{itemize}
\begin{proof}The implications (i)$\Rightarrow$(ii) follows from Corollary 4.4. (ii)$\Rightarrow$(iii) and (i): Since $R$ has finite Kurll dimension and $\Gfd_{R}\H_{\fm}^d(M$) is finite, we have the finitness of $\Gpd_{R}\H_{\fm}^d(M)$ by \cite[Theorem 3.4]{EST}. Hence, by \cite[Corollary 2.4]{EMX}, $\Gpd_{R}\H_{\fm}^d(M)\leq\dim R$. Therefore, in view of \cite[Theorem 4.23]{CF}, we get the following inequalities $$\Gpd_{R}M+d=\Gfd_{R}\H_{\fm}^d(M)\leq\Gpd_{R}\H_{\fm}^d(M)\leq\dim R.$$
Now, one can use \cite[Proposition 2.16 and Theorem 1.25]{CF} to see that $\Gfd_{R}\H_{\fm}^d(M)=\Gpd_{R}\H_{\fm}^d(M)$ and that $\depth M=\dim M$. Thus, $M$ is Cohen-Macaulay and (iii) holds true.

(iii)$\Rightarrow$(ii): First, we notice that, by \cite[Theorem 4.27]{CF}, $\Gfd_{R}\H_{\fm}^d(M)=\Gfd_{\hat{R}}\H_{\fm\hat{R}}^d(\hat{M})$ and that, in view of \cite[propositions 4.23 and 2.20]{CF}, the following inequalities hold: { \[\begin{array}{rl}
\Gfd_{\hat R}\H_{\fm\hat{R}}^d(\hat{M}) &\leq\Gpd_{\hat{R}}\H_{\fm\hat{R}}^d(\hat{M})\\
&\leq\Gpd_{R}\H_{\fm}^d({M}).
\end{array}\]}

 Now, since $\H_{\fm}^d(M)$ is an Artinian $\hat{R}$-module, one can use \cite[Theorem 4.16]{CF} to see that the finitely generated $\hat{R}$-module $\Hom_{{R}}(\H_{\fm}^d(M),\E(R/\fm))$ is of finite Gorenstein injective dimension. Therefore, by \cite[Theorem 3.24]{CF} and \cite[Theorem 4.16]{CF}, $\Gfd_{R}\H_{\fm}^d(M)=\Gid_{\hat{R}}\Hom_{{R}}(\H_{\fm}^d(M),\E(R/\fm))=\dim R$. Hence, one can use \cite[Corollary 2.4]{EMX} and above inequalities to complete the proof.

\end{proof}
\end{prop}

Next, we single out a certain case of Proposition 4.5. Notice that the proof of the following corollary is similar to the proof of Proposition 4.5(ii)$\Rightarrow$(i).
\begin{cor}Suppose that \emph{$\dim R=d$}. Then the following statements are equivalent.
\begin{itemize}
\item[(i)]{\emph{$R$ is Cohen-Macaulay.}}
\item[(ii)]{\emph{$\Gfd_{R}\H_{\fm}^d(R)=d$}}.
\end{itemize}
\end{cor}

The following proposition is a generalization of Theorem 4.3 in terms of $G_{C}$-dimensions.
\begin{prop}Let $n$ be a non-negative integer, $C$ be a semidualizing $R$-module and let $M$ be an $R$-module such that $\emph{H}^{i}_\fa(M)=0$ for all $i\neq n$. Then the following statements hold true.
\begin{itemize}
\item[(i)]{ If \emph{$\GCfd_{R} M<\infty$}, then \emph{$\GCfd_{R}\H^{n}_\fa(M)\leq\GCfd_{R}M+n$.}}
\item[(ii)]{If \emph{$\GCfd_{R}\H^{n}_\fa(M)<\infty$}, then \emph{$\GCfd_R M<\infty$} whenever $M$ is Cohen-Macaulay.}
\end{itemize}
Furthermore, in \emph{(i)} the equality holds whenever $M$ is finitely generated.
\begin{proof} First, we notice that, in view of \cite[Theorem 4.2.1]{MB}, $\H_{\fa}^i(M)\cong\H_{\fa\oplus C}^i(M)$ for all $i$. On the other hand, by using \cite[Theorem 2.16]{HJ}, we have $\GCpd_{R}M=\Gpd_{R\ltimes C}M$ and $\GCfd_{R}\H_{\fa}^n(M)=\Gfd_{R\ltimes C}\H_{\fa}^n(M)=\Gfd_{R\ltimes C}\H_{\fa\oplus C}^n(M)$. Hence, by replacing $R$ with $R\ltimes C$, one can use Theorem 4.3 to complete the proof.

\end{proof}
\end{prop}
The following corollary is a consequence of the pervious proposition and Proposition 4.5.
\begin{cor}Let $R$ be Cohen-Macaulay, $C$ be a semidualizing $R$-module and let $M$ be a $d$-dimensional finitely generated $R$-module of finite $G_{C}$-projective dimension. Then the following statements are equivalent.
\begin{itemize}
\item[(i)]{\emph{$M$ is Cohen-Macaulay.}}
\item[(ii)]{\emph{$\GCfd_{R}\H_{\fm}^d(M)=\GCfd_{R}M+d$}}.
\item[(iii)]{\emph{$\GCpd_{R}\H_{\fm}^d(M)=\GCpd_{R}M+d$}}.
\end{itemize}
\begin{proof}We notice that, by using \cite[Exercise 1.2.26]{BH} and \cite[Theorem 2.2.6]{W}, one can deduce that $(R\ltimes C, \fm\oplus C)$ is a Cohen-Macaulay local ring. Also, $M$ is a Cohen-Macaulay $R$-module if and only if $M$ is a Cohen-Macaulay $R\ltimes C$-module. Therefore, the assertion follows from Proposition 4.5 and Proposition 4.7.
\end{proof}
\end{cor}

%%% -----------------------------------
%%% -----------------------------------
%%% -----------------------------------

%%% -----------------------------------
%%% -----------------------------------
%%% -----------------------------------

%%% -----------------------------------
%%% -----------------------------------
%%% -----------------------------------
\end{document}